\documentclass[12pt, reqno, a4paper]{amsart}

\usepackage{ amssymb, amsmath, enumerate, amsfonts, amsthm, mathrsfs, url, bm, mathtools}

\usepackage{xcolor}  	
\usepackage{hyperref}
\hypersetup{
colorlinks,
   linkcolor={cyan!80!black},
   citecolor={cyan!80!black},
 urlcolor={cyan!80!black}
}

\usepackage{color}

\usepackage[margin=1in]{geometry}

\RequirePackage{doi}

\usepackage{amscd}
\usepackage{amsfonts}
\usepackage{float}
\usepackage{color}
\usepackage[
backend=biber,
style=alphabetic,
]{biblatex}
\usepackage{bookmark}

\renewbibmacro{in:}{}
\DeclareFieldFormat{title}{#1}

\DeclareFieldFormat[article]{title}{\mkbibemph{#1}}       
\DeclareFieldFormat[incollection]{title}{\mkbibemph{#1}}  
\DeclareFieldFormat[book]{title}{\mkbibemph{#1}}          
\DeclareFieldFormat[incollection]{booktitle}{#1}          
\DeclareFieldFormat[article]{journaltitle}{#1}            

\AtEveryBibitem{%
  \ifentrytype{misc}{\DeclareFieldFormat{title}{\mkbibemph{#1}}}{}}

\DeclareFieldFormat{eprint:eprint}{arXiv:\href{https://arxiv.org/abs/#1}{#1}}

\DeclareFieldFormat[inproceedings]{title}{\mkbibemph{#1}}

\DeclareFieldFormat[inproceedings]{booktitle}{#1}

\usepackage{amssymb}

\newtheorem{theorem}{Theorem}[section]
\newtheorem{lemma}{Lemma}[section]
\newtheorem{example}{Example}[section]

\newtheorem{proposition}{Proposition}[section]

\theoremstyle{definition}
\newtheorem{definition}{Definition}[section]

\theoremstyle{remark}

\numberwithin{equation}{section}

\renewcommand{\Re}{\mathrm{Re}}

\newcommand{\C}{{\mathbb C}}
\newcommand{\R}{{\mathbb R}}

\newcommand{\E}{{\mathbb E}}
\newcommand{\N}{{\mathbb N}}
\newcommand{\T}{\mathbb T}
\newcommand{\X}{\mathbb X}

\renewcommand{\leq}{\leqslant}
\renewcommand{\geq}{\geqslant}

\begin{document}

\title[Low moments of automorphic random multiplicative function sums]
{Low moments of automorphic random multiplicative function sums}

\author{Sun-Kai Leung}
\address{Mathematical Institute, University of Oxford, Andrew Wiles Building\\
Radcliffe Observatory Quarter, Woodstock Road\\
Oxford OX2 6GG\\
United Kingdom}
\email{sunkaileung@gmail.com}


\date{}

\begin{abstract}
We determine the order of the low moments of partial sums of random
multiplicative functions associated with Euler products of bounded degree
whose unitary local parameters satisfy a prime-square cancellation condition,
thereby generalizing Harper's seminal work.  The result
applies to irreducible compact-group representations of unitary or symplectic
type, including the higher-rank Sato--Tate model $SU(d)$ and the odd
symmetric-power Sato--Tate model $\operatorname{Sym}^{m}SU(2)$.
\end{abstract}

\maketitle

\section{Introduction}

\textit{Random multiplicative functions} (RMFs) provide probabilistic models
for families of multiplicative functions, including Archimedean characters
$n\mapsto n^{it}$ for large real numbers $t$ and Dirichlet characters modulo a large prime $p$ in the
Steinhaus case (see, for instance, \cite{MR1815216,harper2025bettersquarerootcancellationnumber}).
Similarly, Sato--Tate RMFs model Hecke eigenvalues under harmonic averaging, as in \cite{MR2035301,MR3918448} and the recent
works \cite{hamdan2026lowmomentsheckeeigenvalue,xu2026typicalsizeheckeeigenvalue}.
Motivated by automorphic $L$-functions of higher degree, we
consider the following class of RMFs.

\begin{definition}[Automorphic random multiplicative function]\label{def:armf}
Given $d\in\N,$ let
\begin{align*}
 \boldsymbol\alpha(p)=(\alpha_1(p),\ldots,\alpha_d(p))\in\T^d
\end{align*}
be independent and identically distributed random vectors as $p$ varies over
the primes, where $\T:=\{z\in\C:|z|=1\},$ and define the complete symmetric
polynomials $h_k$ by
\begin{align*}
 \prod_{j=1}^d(1-\alpha_jz)^{-1}
 =\sum_{k\geq0}h_k(\boldsymbol\alpha)z^k.
\end{align*}
We set $\X(p^k):=h_k(\boldsymbol\alpha(p))$ and extend $\X$
multiplicatively, with $\X(1)=1.$  We call $\X$ an \textit{automorphic
random multiplicative function of degree $d$}.
\end{definition}

The terminology refers to the formal Euler-product identity
\begin{align*}
 \sum_{n\geq1}\frac{\X(n)}{n^s}
 =\prod_p\prod_{j=1}^d\left(1-\frac{\alpha_j(p)}{p^s}\right)^{-1},
\end{align*}
in which the random variables $\alpha_j(p)$ play the role of the Satake parameters
of a tempered automorphic $L$-function, without asserting automorphy of an
individual sample.

Harper~\cite{HarperLow} determined the order of magnitude of the low moments
of partial sums of Steinhaus RMFs.  We extend his result to automorphic RMFs
for which the complete symmetric polynomials $h_1,h_2$ are centered and
$h_1$ has unit variance.

\begin{theorem}\label{thm:main}
Given $d\in\N,$ let $\X$ be an automorphic random multiplicative function of
degree $d$ satisfying
\begin{equation}\label{eq:conditions}
 \E[h_1(\boldsymbol\alpha)]=0,\qquad
 \E[|h_1(\boldsymbol\alpha)|^2]=1,\qquad
 \E[h_2(\boldsymbol\alpha)]=0.
\end{equation}
Then uniformly for all sufficiently large $x$ and $0\leq q\leq1,$ we have
\begin{align*}
 \E\left[\left|\sum_{n\leq x}\X(n)\right|^{2q}\right]
 \asymp_d
 \left(\frac{x}{1+(1-q)\sqrt{\log\log x}}\right)^q.
\end{align*}
The implied constants and the lower threshold for $x$ depend only on $d,$
uniformly over the local laws satisfying \eqref{eq:conditions}.
\end{theorem}

The main new ingredient in our argument is a two-sided transference principle for low moments
of Euler-product masses.  A prime-by-prime Gaussian replacement, combined with
Kahane's inequality, permits comparison with the Steinhaus model for the upper
bound and with the Rademacher model for the lower bound.  This approach is
inspired by the Gaussian comparison methods of Saksman--Webb
\cite{SaksmanWebb} and Gorodetsky--Wong \cite[Section~3]{GW2025}.  Related
blockwise Gaussian approximations for the degree-two Sato--Tate model appear in
\cite[Section~5]{hamdan2026lowmomentsheckeeigenvalue}.

For the upper bound, the initial reduction of Gorodetsky--Wong
\cite[Lemma~1.2 and Section~2]{GW2025} reduces the problem to an Euler-product
integral, which we transfer to the Steinhaus model and estimate using Harper's
Euler-product bound in the formulation of Hamdan
\cite[Proposition~1.2]{Hamdan}.  For the lower bound, we adapt Harper's
large-prime reduction \cite[Propositions~3--4]{HarperLow} and transfer the
resulting Euler-product integral to the Rademacher model, where Harper's barrier
estimates \cite[Sections~5.2--5.3]{HarperLow} apply.  The identities in
\eqref{eq:conditions} provide the moment matching and bounded covariance errors
required in both comparisons.

Compact-group models furnish natural examples. Given a compact group $G$
and a nontrivial irreducible unitary representation $V$, we may choose an
independent Haar-distributed element of $G$ at every prime and use its
eigenvalues on $V$ as the local parameters.  The character identities,
together with the orthogonality relations, give
\begin{align*}
 \E[h_1(\boldsymbol\alpha)]=0,\qquad
 \E[|h_1(\boldsymbol\alpha)|^2]=1,\qquad
 \E[h_2(\boldsymbol\alpha)]=\dim(\operatorname{Sym}^2V)^G.
\end{align*}
Therefore, Theorem~\ref{thm:main} applies whenever $V$ has no invariant
symmetric bilinear form, as happens for representations of \textit{unitary} or
\textit{symplectic} type, whose definitions we recall in Section~\ref{sec:examples}.

The standard representations of the unitary groups $U(d)$ extend the Steinhaus
model to arbitrary rank.

\begin{example}[Higher-rank Steinhaus RMF]\label{ex:steinhaus}
Let $G=U(d)$, let $V=\C^d$ be its standard representation, and, at every
prime, let $\boldsymbol\alpha(p)$ be the unordered spectrum of an independent
Haar-distributed matrix in $U(d)$.  The representation $V$ is irreducible,
of unitary type, and of dimension $d,$ and hence \eqref{eq:conditions} holds.
When $d=1,$ we have
$h_k(\alpha)=\alpha^k$ with $\alpha$ uniform on $\T,$ and
Theorem~\ref{thm:main} recovers \cite[Theorem~1]{HarperLow}.
\end{example}

The irreducible representations of the special unitary group $SU(2)$ give the higher symmetric-power
models associated with the usual Sato--Tate law.

\begin{example}[Odd symmetric-power Sato--Tate RMF]\label{ex:odd-sym}
Let $G=SU(2)$, let $V=\operatorname{Sym}^m(\C^2)$ for odd $m,$ and, at every
prime, let $\boldsymbol\alpha(p)$ be the unordered spectrum of an independent
Haar-distributed element of $SU(2)$ acting on $V$.  This spectrum is
\begin{align*}
 e^{i(m-2j)\theta_p}\qquad (0\leq j\leq m),
\end{align*}
where the Sato--Tate angles $\theta_p$ are independent with density
$(2/\pi)\sin^2\theta$ on $[0,\pi].$  The representation $V$ is irreducible,
of symplectic type, and of dimension $m+1,$ and hence
\eqref{eq:conditions} holds.  When $m=1,$ this gives the usual Sato--Tate
RMF, for which Theorem~\ref{thm:main} recovers
the bounds of Xu--Zheng
\cite[Theorem~1.8]{xu2026typicalsizeheckeeigenvalue}, with the upper bound
also obtained in \cite[Theorem~1.1]{hamdan2026lowmomentsheckeeigenvalue}.
\end{example}

The special unitary groups $SU(d)$ give models associated with harmonic averages
of Hecke--Maass cusp forms.

\begin{example}[Higher-rank Sato--Tate RMF]\label{ex:higher-rank-st}
Let $G=SU(d)$ for $d\geq3,$ let $V=\C^d$ be its standard representation,
and, at every prime, let $\boldsymbol\alpha(p)$ be the unordered spectrum of
an independent Haar-distributed matrix in $SU(d)$.  The representation $V$
is irreducible, of unitary type, and of dimension $d,$ and hence
\eqref{eq:conditions} holds.  Theorem~\ref{thm:main} therefore applies to
this model for every $d\geq3.$
\end{example}

Jana \cite[Theorem~2]{Jana} proved weighted fixed-prime equidistribution
towards this higher-rank Sato--Tate law in the analytic-conductor aspect, using the Kuznetsov trace formula and the identification of prime-power Fourier coefficients
with Schur characters (see \cite[Theorem~1, Lemma~4.1 and Section~4.1]{Jana}).
Notice that his average includes the continuous spectrum, while a cuspidal variant after
fixing a supercuspidal component is explained in
\cite[Remark~1 and Theorem~7]{Jana}.  Applying the same character argument at
finitely many primes gives the product law, which motivates the
independence in Definition~\ref{def:armf}.  For the corresponding result in
the Laplace-eigenvalue aspect, we refer to \cite[Theorem~1.2]{Zhou}, whose
formulation assumes harmonic orthogonality for $d\geq4.$

The centering of $h_2$ is not automatic, as the even symmetric-power
representations of $SU(2)$ are of orthogonal type and do not satisfy the last
identity in \eqref{eq:conditions}.  Section~\ref{sec:examples} gives an exact
calculation for the symmetric-square model.\\

\noindent\textit{Notation.}
Throughout the paper, we use the standard big $O$ and little $o$
notations, the Vinogradov notation $\ll,\gg,$ and the Hardy notation $\asymp,$ where the
implied constants depend only on the subscripted parameters unless otherwise
specified. Let $P^+(n)$ denote the largest prime factor of $n,$ with
$P^+(1)=1.$  We write
\begin{align*}
 S(x)&:=\sum_{n\leq x}\X(n) \qquad(x\geq0),\\
 T_y(u)&:=\sum_{\substack{n\leq u\\P^+(n)\leq y}}\X(n)
       \qquad(u\geq0,\ y\geq2),\\
 F_y(s)&:=\prod_{p\leq y}\prod_{j=1}^d
              (1-\alpha_j(p)p^{-s})^{-1}
              \qquad(\Re(s)>0,\ y\geq2).
\end{align*}
We also denote $F_{(z,y]}(s):=F_y(s)/F_z(s)$ for $2\leq z\leq y,$ and
$D_q(y):=1+(1-q)\sqrt{\log\log y}$ for $y\geq3.$

\section{Preparation}

To prove Theorem~\ref{thm:main}, we require several preparatory lemmas.  First,
since the higher prime-power coefficients are not necessarily orthogonal, we use the
following covariance expansion to control the resulting error terms.

\begin{lemma}\label{lem:covariance}
Let $m,n\in\N.$  There exists a multiplicative function $g:\N^2\to\C,$ in
the sense that
\begin{align*}
 g(u_1u_2,v_1v_2)=g(u_1,v_1)g(u_2,v_2)
\end{align*}
whenever $(u_1v_1,u_2v_2)=1,$ and it satisfies
\begin{equation}\label{eq:covariance}
 \E[\X(m)\overline{\X(n)}]
 =\sum_{r\mid(m,n)}g(m/r,n/r).
\end{equation}
Also, for every $\sigma>1/3,$ we have
\begin{equation}\label{eq:kernel-norm}
 \sum_{u,v\geq1}\frac{|g(u,v)|}{(uv)^\sigma}\ll_{d,\sigma}1.
\end{equation}
Consequently, for every $0\leq a<b,$ we have
\begin{align}
 \E\left[\left|\sum_{a<n\leq b}\X(n)\right|^2\right]
 \ll_d b-a+b^{5/6}. \label{eq:short-second}
\end{align}
The same estimate holds when the summand is restricted to integers whose
prime factors belong to a fixed set of primes.
\end{lemma}

\begin{proof}
For integers $k,l \geq 0,$ put
\begin{align*}
 c(k,l):=\E[h_k(\boldsymbol\alpha)
                 \overline{h_l(\boldsymbol\alpha)}],
\end{align*}
where $c(k,l)$ is understood to be zero if either index is negative, and
define
\begin{align*}
 g(p^k,p^l):=c(k,l)-c(k-1,l-1).
\end{align*}
Extending this definition multiplicatively and telescoping at each prime
proves \eqref{eq:covariance}.

The assumptions in \eqref{eq:conditions} give $g(1,1)=1$ and
\begin{align*}
 g(p^k,p^l)=0\qquad\text{whenever }1\leq k+l\leq2.
\end{align*}
Using the bound $|h_k(\boldsymbol\alpha)|\leq\binom{k+d-1}{d-1},$ we obtain
\begin{align*}
 \sum_{k,l\geq0}|g(p^k,p^l)|p^{-\sigma(k+l)}
 =1+O_{d,\sigma}(p^{-3\sigma}).
\end{align*}
The corresponding Euler product converges when $\sigma>1/3,$ which proves
\eqref{eq:kernel-norm} with constants uniform in the local law.

On substituting $m=ur$ and $n=vr$ in the second moment of an interval sum,
the number of admissible values of $r$ is at most
\begin{align*}
 \frac{b-a}{\max\{u,v\}}+1.
\end{align*}
The first term contributes $O_d(b-a)$ by \eqref{eq:kernel-norm} with
$\sigma=1/2.$  Since $u,v\leq b$ for the second term, we have
\begin{align*}
 \sum_{u,v\leq b}|g(u,v)|
 \leq b^{5/6}\sum_{u,v\geq1}\frac{|g(u,v)|}{(uv)^{5/12}}
 \ll_d b^{5/6},
\end{align*}
which proves \eqref{eq:short-second}. Restricting
to a set of primes removes terms from the same absolute majorants, which
proves the final assertion.
\end{proof}

We pass between smooth sums and Euler products by means of the following
Mellin form of Parseval's identity.

\begin{lemma}\label{lem:mellin}
Let $\eta\geq0.$ Then
\begin{align*}
 \int_1^\infty |T_y(u)|^2\frac{du}{u^{2+2\eta}}
 =\frac1{2\pi}\int_{\R}
   \left|\frac{F_y(1/2+\eta+it)}{1/2+\eta+it}\right|^2dt.
\end{align*}
\end{lemma}

\begin{proof}
Applying \cite[Theorem~5.4]{MR2378655} to the absolutely convergent
Dirichlet series for $F_y(s)$ in $\Re(s)>0$ gives the identity.
\end{proof}

We next compare the Euler product with a Gaussian field prime by prime, using
a replacement argument whose covariance comparison will be applied in both
directions.  Put $\rho:=\E[h_1(\boldsymbol\alpha)^2].$
Notice that the Cauchy--Schwarz inequality and \eqref{eq:conditions} give
$|\rho|\leq1.$

Throughout this section, for a finite set of primes $\mathcal P,$ let
\begin{align*}
 F_{\mathcal P}(s)&:=\prod_{p\in\mathcal P}\prod_{j=1}^d
                     (1-\alpha_j(p)p^{-s})^{-1},\\
 H_{\mathcal P,\rho,\sigma}(t)&:=2\Re\sum_{p\in\mathcal P}
                    g_pp^{-\sigma-it},\\
 m_{\mathcal P,\sigma}&:=
        \exp\left(\sum_{p\in\mathcal P}p^{-2\sigma}\right),
\end{align*}
where the $g_p$ are independent centered complex Gaussian random variables satisfying
\begin{align*}
 \E[|g_p|^2]=1,\qquad \E[g_p^2]=\rho.
\end{align*}
They exist for every $|\rho|\leq1,$ where we permit a degenerate
real covariance matrix when $|\rho|=1.$

\begin{lemma}\label{lem:gaussian-replacement}
There exists a constant $c_d\geq1$ such that for every finite positive Borel
measure $\nu$ on $\R,$ every $\sigma\geq1/2$ and every $0<q\leq1,$ the quantities
\begin{align*}
 \mathcal F_{\mathcal P}&:=\E\left[\left(\int_{\R}
          |F_{\mathcal P}(\sigma+it)|^2\,d\nu(t)\right)^q \right],\\
 \mathcal G_{\mathcal P}&:=m_{\mathcal P,\sigma}^q \,
 \E\left[\left(\int_{\R}
 e^{H_{\mathcal P,\rho,\sigma}(t)
       -\frac{1}{2}\E[H_{\mathcal P,\rho,\sigma}(t)^2]}\,d\nu(t)\right)^q \right]
\end{align*}
satisfy
\begin{align*}
 c_d^{-q} \,\mathcal G_{\mathcal P}\leq\mathcal F_{\mathcal P}
 \leq c_d^q \,\mathcal G_{\mathcal P}.
\end{align*}
\end{lemma}

\begin{proof}
We may assume that $\nu(\R)>0$ and replace one prime at a time, conditioning
on all the remaining local random variables.  After absorbing their positive
local factors into $\nu,$ put $\mu:=\nu/\nu(\R),$ since the normalization
contributes the same factor $\nu(\R)^q$ to both sides of the comparison.  Put
\begin{align*}
 \mathcal I_p:=\int_{\R}p^{-it}\,d\mu(t).
\end{align*}
The two local quantities to be compared are
\begin{align*}
 \mathcal F_p&:=\int_{\R}
 \left|\prod_{j=1}^d(1-\alpha_jp^{-\sigma-it})^{-1}\right|^2d\mu(t),\\
 \mathcal G_p&:=\int_{\R}
 \exp\left(2p^{-\sigma}\Re(g_pp^{-it})
       -p^{-2\sigma}\Re(\rho p^{-2it})\right)d\mu(t).
\end{align*}
Uniformly in $\mu,$ expansion of the local Euler factor gives
\begin{align*}
 \mathcal F_p
=1&+2p^{-\sigma}
       \Re\left(h_1(\boldsymbol\alpha)\mathcal I_p\right)\\
 &+p^{-2\sigma}\left\{|h_1(\boldsymbol\alpha)|^2
       +2\Re\left(h_2(\boldsymbol\alpha)
       \int_{\R}p^{-2it}\,d\mu(t)\right)\right\}
       +O_d(p^{-3\sigma}).
\end{align*}
Applying Taylor's theorem and the identities in \eqref{eq:conditions}, for sufficiently large $p,$ we obtain
\begin{align*}
 \E[\mathcal F_p^q]
={}&1+qp^{-2\sigma}
 +2q(q-1)p^{-2\sigma}
 \E[\left(\Re\left(
 h_1(\boldsymbol\alpha)\mathcal I_p\right)\right)^2]
 +O_d(qp^{-3\sigma}).
\end{align*}
The remainder here is uniform since
\begin{align*}
 |h_k(\boldsymbol\alpha)|
 \leq\binom{k+d-1}{d-1}.
\end{align*}

We now perform the corresponding calculation for $\mathcal G_p.$  Expanding
the exponential and integrating against $\mu,$ we have
\begin{align*}
 \mathcal G_p
=1+2p^{-\sigma}\Re(g_p\mathcal I_p)
 &+p^{-2\sigma}\int_{\R}
 \left\{2(\Re(g_pp^{-it}))^2
       -\Re(\rho p^{-2it})\right\}d\mu(t) \\
 &+O\bigl(p^{-3\sigma}(1+|g_p|)^3
       \exp(p^{-\sigma}(2|g_p|+1))\bigr).
\end{align*}
The Gaussian moment identities give
\begin{gather*}
 2\int_{\R}\E[(\Re(g_pp^{-it}))^2]\,d\mu(t)
 -\Re\left(\rho\int_{\R}p^{-2it}\,d\mu(t)\right)=1,\\
 \E[(\Re(g_p\mathcal I_p))^2]
 =\E[\left(\Re\left(
 h_1(\boldsymbol\alpha)\mathcal I_p\right)\right)^2].
\end{gather*}
Applying Taylor's theorem again and using the Gaussian exponential
moments, we obtain
\begin{align*}
 \E[\mathcal G_p^q]
={}&1+qp^{-2\sigma}
 +2q(q-1)p^{-2\sigma}
 \E[\left(\Re\left(
 h_1(\boldsymbol\alpha)\mathcal I_p\right)\right)^2]
 +O_d(qp^{-3\sigma})\\
={}&\E[\mathcal F_p^q]+O_d(qp^{-3\sigma}).
\end{align*}
For the remainder, let $\mathcal G_p(u)$ denote the defining integral
with $p^{-\sigma}$ replaced by $u.$  For $0\leq u\leq1/2$ and $1\leq k\leq3,$ differentiating the integrand gives
\begin{gather*}
 |\mathcal G_p^{(k)}(u)|\ll(1+|g_p|)^k\mathcal G_p(u),
 \qquad
 \mathcal G_p(u)^q\leq\exp(|g_p|+1),\\
 (\mathcal G_p^q)'''
 =q\mathcal G_p^{q-1}\mathcal G_p'''
 +3q(q-1)\mathcal G_p^{q-2}\mathcal G_p'\mathcal G_p''
 +q(q-1)(q-2)\mathcal G_p^{q-3}(\mathcal G_p')^3.
\end{gather*}
Consequently, we have
\begin{gather*}
 \E\left[\sup_{0\leq u\leq1/2}
 |(\mathcal G_p(u)^q)'''|\right]
 \ll q\E[(1+|g_p|)^3\exp(|g_p|+1)]\ll q,
\end{gather*}
which gives the stated remainder uniformly for $|\rho|\leq1.$  Since the common main term in the two expansions
is $1+O_d(qp^{-2\sigma}),$ it follows that, for every sufficiently large $p,$
we have
\begin{align}\label{eq:local-ratio}
 \exp(-c_dqp^{-3\sigma})
 \leq\frac{\E[\mathcal F_p^q]}
              {\E[\mathcal G_p^q]}
 \leq \exp(c_dqp^{-3\sigma}).
\end{align}
For the finitely many remaining primes, we have
\begin{align*}
 (1+p^{-\sigma})^{-2d}\leq\mathcal F_p
 \leq(1-p^{-\sigma})^{-2d},
\end{align*}
and
\begin{align*}
 \exp(-qp^{-2\sigma})\leq\E[\mathcal G_p^q]
 \leq \exp(qp^{-2\sigma}),
\end{align*}
where the second estimate follows from Jensen's inequality and the Gaussian exponential formula. After multiplying over these primes, their contribution
lies between $c_d^{-q}$ and $c_d^q.$

Finally, it follows from the identity
\begin{align*}
 \frac12 \, \E[H_{\mathcal P,\rho,\sigma}(t)^2]
 =\sum_{p\in\mathcal P}p^{-2\sigma}
  +\Re\left(\rho\sum_{p\in\mathcal P}p^{-2\sigma-2it}\right)
\end{align*}
that
\begin{align*}
 m_{\mathcal P,\sigma}\int_{\R}
 e^{H_{\mathcal P,\rho,\sigma}(t)
       -\frac12\E[H_{\mathcal P,\rho,\sigma}(t)^2]}
 \,d\nu(t)
=\int_{\R}\prod_{p\in\mathcal P}
 e^{2p^{-\sigma}\Re(g_pp^{-it})
       -p^{-2\sigma}\Re(\rho p^{-2it})}\,d\nu(t).
\end{align*}
Iterating \eqref{eq:local-ratio}, we obtain
\begin{align*}
 c_d^{-q}\exp\left(-c_dq\sum_{p\in\mathcal P}p^{-3\sigma}\right)
 \leq\frac{\mathcal F_{\mathcal P}}{\mathcal G_{\mathcal P}}
 \leq c_d^q\exp\left(c_dq\sum_{p\in\mathcal P}p^{-3\sigma}\right).
\end{align*}
Since $\sigma\geq1/2$ and $\sum_pp^{-3/2}<\infty,$ the result follows after
enlarging $c_d.$
\end{proof}

The covariance kernels in our applications agree only up to a bounded
function.  The following consequence of Kahane's comparison
inequality permits such a perturbation.

\begin{lemma} \label{lem:kahane}
Let $G_1,G_2$ be centered continuous Gaussian fields on a compact
interval $I,$ and suppose that
\begin{align*}
 \left|\E[G_1(t)G_1(u)]
       -\E[G_2(t)G_2(u)]\right|\leq c
\end{align*}
for every $t,u\in I.$  Given a finite positive measure $\nu$ on $I,$ put
\begin{align*}
 \mathcal M_j:=\int_I
 e^{G_j(u)-\frac{1}{2}\E[G_j(u)^2]}\,d\nu(u)
 \qquad(j=1,2).
\end{align*}
Then for $0<q\leq1,$ we have
\begin{align*}
 e^{-cq(1-q)/2}\E[\mathcal M_2^q]
 \leq\E[\mathcal M_1^q]
 \leq e^{cq(1-q)/2}\E[\mathcal M_2^q].
\end{align*}
\end{lemma}

\begin{proof}
Let $\mathcal{N}$ be an independent standard real Gaussian random variable.  Since
\begin{align*}
 \E[G_1(t)G_1(u)]\leq\E[G_2(t)G_2(u)]+c,
\end{align*}
Kahane's inequality \cite[Lemme~1]{Kahane}, applied to the concave function
$x\mapsto x^q$ and the fields $G_1$ and $G_2+\sqrt c\,\mathcal{N},$ gives
\begin{align*}
 \E[\mathcal M_1^q]
 &\geq\E[e^{q\sqrt c\,\mathcal{N}-qc/2}]\, \E[\mathcal M_2^q]\\
 &=e^{-cq(1-q)/2}\, \E[\mathcal M_2^q].
\end{align*}
Interchanging the two fields proves the other inequality, with the integral
form following by approximation by finite positive sums.
\end{proof}

We also record the prime-sum estimate used to compare the Gaussian
covariances at different spatial scales.

\begin{lemma}\label{lem:prime-sum}
There exists an absolute constant $c>0$ such that uniformly for $2\leq a\leq b$
and $v\in\R,$ we have
\begin{equation}\label{eq:prime-covariance-integral}
 \sum_{e^a<p\leq e^b}p^{-1-iv}
 =\int_a^b e^{-ivw} \,\frac{dw}{w}
 +O\left((1+|v|)(1+\sqrt a)e^{-c\sqrt a}\right).
\end{equation}
If $v\ne0,$ then
\begin{align*}
\int_a^b e^{-ivw} \,\frac{dw}{w}
 \ll (|v|a)^{-1}.
\end{align*}
\end{lemma}

\begin{proof}
Applying the prime number theorem with the de la Vall\'ee Poussin error term followed by partial summation, we obtain
\begin{gather*}
 \left|\sum_{e^a<p\leq e^b}p^{-1-iv}
       -\int_a^b\exp(-ivw)\frac{dw}{w}\right|
 \ll(1+|v|)\left(\exp(-c\sqrt a)
       +\int_a^\infty\exp(-c\sqrt w)\,dw\right)\\
 \ll(1+|v|)(1+\sqrt a)\exp(-c\sqrt a),
\end{gather*}
where the last estimate follows on substituting $w=u^2.$  This proves
\eqref{eq:prime-covariance-integral}, and one integration by parts gives
the second assertion.
\end{proof}

We shall also use the following second-moment estimate for smooth sums, which
follows from the covariance expansion in Lemma~\ref{lem:covariance}.

\begin{lemma}\label{lem:smooth-second}
Uniformly for sufficiently large $y$ and $u\geq1,$ we have
\begin{align*}
 \E[|T_y(u)|^2]
 \ll_d u(\log y)\exp\left(-\frac{\log u}{\log y}\right).
\end{align*}
\end{lemma}

\begin{proof}
Put $\eta=1/\log y.$  Rankin's argument and Mertens' estimate give
\begin{align*}
 \Psi(w,y):=&\,\#\{n\leq w:P^+(n)\leq y\}\\
 \leq &\, w^{1-\eta}\prod_{p\leq y}(1-p^{-1+\eta})^{-1}
 \ll w^{1-\eta}\log y.
\end{align*}
Here the change from $p^{-1}$ to $p^{-1+\eta}$ costs a bounded factor,
since $p^\eta-1\ll\eta\log p$ for $p\leq y$ and
\begin{align*}
\eta\sum_{p\leq y}(\log p)/p\ll1.
\end{align*}

Using \eqref{eq:covariance} and discarding restrictions on $v$ and $w,$
we obtain
\begin{align*}
 \E[|T_y(u)|^2]
 &\leq\sum_{v,w\geq1}|g(v,w)|
       \Psi\left(\frac{u}{\max\{v,w\}},y\right)\\
 &\ll u^{1-\eta}(\log y)
       \sum_{v,w\geq1}\frac{|g(v,w)|}
       {\max\{v,w\}^{1-\eta}}.
\end{align*}
For sufficiently large $y,$ we have
$\max\{v,w\}^{1-\eta}\geq(vw)^{5/12},$ and applying
\eqref{eq:kernel-norm} with $\sigma=5/12$ proves the lemma.
\end{proof}

\section{The upper Euler-product estimate}

For a Steinhaus random multiplicative function $f,$ let
\begin{align*}
 F_Y^{\mathrm{St}}(s):=\prod_{p\leq Y}(1-f(p)p^{-s})^{-1}.
\end{align*}
Harper's Euler-product estimate, in the formulation of
Hamdan \cite[Proposition~1.2]{Hamdan}, states that uniformly for sufficiently
large $Y$ and $0\leq q\leq1,$ we have
\begin{align} \label{eq:hamdan-input}
 \E\left[\left(\int_0^1
 |F_Y^{\mathrm{St}}(1/2+it)|^2dt\right)^q\right]
 \ll
 \left(\frac{\log Y}{D_q(Y)}\right)^q.
\end{align}

The term involving the pseudo-variance $\rho$ is not stationary, although it
is bounded on fixed compact intervals away from the origin.  We therefore
remove the small primes at a scale depending on the interval and then compare
with the circular Gaussian field underlying \eqref{eq:hamdan-input}.

\begin{lemma}\label{lem:local-upper-general}
Fix a sufficiently large absolute constant $c_0,$ and put $L=\log y$
for sufficiently large $y.$  Then the following estimates hold.

\begin{enumerate}
\item[\rm (i)] If $L^{-1/2}\leq h\leq1$ and $z=\exp(c_0/h),$ then uniformly
for $2/3\leq q\leq1,$ we have
\begin{align*}
 \E\left[\left(\int_h^{2h}|F_y(1/2+it)|^2dt\right)^q\right]
 \ll_d\left(\frac{hL}{D_q(y)}\right)^q.
\end{align*}
The same estimate holds on $[-2h,-h].$

\item[\rm (ii)] Put $b_n=c_0\log^2(n+3)$ for $n\in \N.$  If
$b_n\leq L^{1/2},$ then uniformly for $2/3\leq q\leq1,$ we have
\begin{align*}
 \E\left[\left(\int_n^{n+1}|F_y(1/2+it)|^2dt\right)^q\right]
 \ll_d b_n^{1-q}\left(\frac{L}{D_q(y)}\right)^q.
\end{align*}
The same estimate holds with $[n,n+1]$ replaced by $[-n-1,-n].$
\end{enumerate}
\end{lemma}

\begin{proof}
The local expansion furnished by \eqref{eq:conditions}, together with
Mertens' estimate, gives uniformly in $t$ that
\begin{equation}\label{eq:prime-second-uniform}
 \E[|F_z(1/2+it)|^2]\asymp_d\log z.
\end{equation}
Indeed, each local expectation equals $1+p^{-1}+O_d(p^{-3/2}).$  On
conditioning on the primes in $(z,y],$ for every interval $I,$ Jensen's inequality and
\eqref{eq:prime-second-uniform} give
\begin{equation}\label{eq:low-prime-removal}
 \E\left[\left(\int_I|F_y(1/2+it)|^2dt\right)^q\right]
 \ll_d(\log z)^q \,
 \E\left[\left(\int_I|F_{(z,y]}(1/2+it)|^2dt\right)^q\right].
\end{equation}

We prove both parts by the same comparison, taking $b=c_0/h$ in
part~\textup{(i)} and $b=b_n$ in part~\textup{(ii)}, with $z=\exp(b).$
Partition the interval into subintervals $I=[t_0,t_0+\delta]$ with
$\delta\leq b^{-1},$ put $J=[0,b\delta]\subseteq[0,1],$ and write
$t=t_0+u/b.$  For
\begin{align*}
 \widetilde H(u):=H_{\{p:z<p\leq y\},\rho,1/2}(t_0+u/b),
\end{align*}
we have
\begin{align*}
 \E[\widetilde H(u)\widetilde H(v)]
={}&2\sum_{z<p\leq y}\frac{\cos((u-v)\log p/b)}p
 +2\Re\left(\rho\sum_{z<p\leq y}
 p^{-1-i(2t_0+(u+v)/b)}\right).
\end{align*}
Applying Lemma~\ref{lem:prime-sum}, we obtain for the first term
\begin{align*}
 2\int_1^{L/b}\cos((u-v)w)\frac{dw}{w}+O(1),
\end{align*}
which differs by $O(1)$ from
\begin{align*}
 \E[H_{\{p:p\leq Y\},0,1/2}(u)
              H_{\{p:p\leq Y\},0,1/2}(v)],
 \qquad Y:=\exp(L/b).
\end{align*}
This is the circular Gaussian field associated with a Steinhaus Euler
product.  The second term is also $O(1)$, since in part~\textup{(i)} its
frequency has modulus at least $2h$ and $hb=c_0,$ while in
part~\textup{(ii)} its modulus is comparable with $n+1$ and
\begin{align*}
 (n+1)(1+\sqrt{b_n})e^{-c\sqrt{b_n}}\ll1
\end{align*}
on choosing $c_0$ sufficiently large.  Thus the two covariance kernels
differ by $O(1)$ for all $u,v\in J,$ including on the diagonal.  Moreover,
Mertens' estimate gives
\begin{align*}
 m_{\{p:z<p\leq y\},1/2}\asymp L/b
 \asymp m_{\{p:p\leq Y\},1/2}.
\end{align*}
Applying Lemma~\ref{lem:gaussian-replacement}, followed by the change of
variables $t=t_0+u/b$ and Lemma~\ref{lem:kahane}, and then applying
Lemma~\ref{lem:gaussian-replacement} to the Steinhaus model, we obtain
\begin{gather*}
 \E\left[\left(\int_I
     |F_{(z,y]}(1/2+it)|^2dt\right)^q\right]\\
\asymp_d b^{-q}m_{\{p:z<p\leq y\},1/2}^q \,
 \E\left[\left(\int_J
 e^{\widetilde H(u)-\frac12\E[\widetilde H(u)^2]}du
 \right)^q\right]\\
\ll_d b^{-q}m_{\{p:z<p\leq y\},1/2}^q \,
 \E\left[\left(\int_J
 e^{H_{\{p:p\leq Y\},0,1/2}(u)
 -\frac12\E[H_{\{p:p\leq Y\},0,1/2}(u)^2]}du
 \right)^q\right]\\
\asymp_d b^{-q}
 \left(\frac{m_{\{p:z<p\leq y\},1/2}}
 {m_{\{p:p\leq Y\},1/2}}\right)^q
 \E\left[\left(\int_J
 |F_Y^{\mathrm{St}}(1/2+iu)|^2du\right)^q\right].
\end{gather*}
Since $J\subseteq[0,1]$ and the two normalizing factors are comparable,
applying \eqref{eq:hamdan-input} yields
\begin{align*}
 \E\left[\left(\int_I|F_{(z,y]}(1/2+it)|^2dt\right)^q\right]
 \ll_d\left(\frac{L}{b^2D_q(Y)}\right)^q.
\end{align*}
Since $b\leq c_0L^{1/2},$ we have $D_q(Y)\asymp D_q(y)$ for sufficiently
large $y,$ and using \eqref{eq:low-prime-removal} gives
\begin{align*}
 \E\left[\left(\int_I|F_y(1/2+it)|^2dt\right)^q\right]
 \ll_d\left(\frac{L}{bD_q(y)}\right)^q.
\end{align*}
Using the sub-additivity $(\sum_j u_j )^q\leq\sum_j u_j^q$ for $u_j\geq0,$
we sum over $O(1)$ subintervals in part~\textup{(i)} and $O(b_n)$ subintervals in
part~\textup{(ii)} to prove both assertions, while
on the intervals $[-2h,-h]$ and $[-n-1,-n]$ the sum-frequency remains
bounded away from zero and the same argument applies.
\end{proof}

The preceding local estimates cover the whole line after a dyadic
decomposition near the origin and a unit decomposition at large heights.

\begin{proposition}\label{prop:general-mass-upper}
Uniformly for sufficiently large $y$ and $0\leq q\leq1,$ we have
\begin{align*}
 \E\left[\left(\int_{\R}
 \left|\frac{F_y(1/2+it)}{1/2+it}\right|^2dt\right)^q\right]
 \ll_d\left(\frac{\log y}{D_q(y)}\right)^q.
\end{align*}
\end{proposition}

\begin{proof}
Put $L=\log y$ and first suppose that $2/3\leq q\leq1.$  By
\eqref{eq:prime-second-uniform} and Jensen's inequality, the interval
$|t|\leq2L^{-1/2}$ contributes at most $O_d(L^{q/2}),$ which is bounded by
$O_d((L/D_q(y))^q).$  The remaining part of $[-1,1]$ is covered by intervals
of the form $[2^{-r},2^{1-r}]$ and their reflections, where
$2^{-r}\geq L^{-1/2}.$  Lemma~\ref{lem:local-upper-general}\textup{(i)} and
the convergence of $\sum_{r\geq0}2^{-rq}$ give a total contribution
$O_d((L/D_q(y))^q).$

For $|t|\geq1,$ use the unit intervals and
Lemma~\ref{lem:local-upper-general}\textup{(ii)} whenever
$b_n\leq L^{1/2}.$  After including the factor
$|1/2+it|^{-2},$ their total contribution is
\begin{align*}
 \ll_d\left(\frac{L}{D_q(y)}\right)^q
 \sum_{n\geq1}n^{-2q}\{\log(n+3)\}^{2(1-q)}
 \ll_d\left(\frac{L}{D_q(y)}\right)^q.
\end{align*}
The estimate is uniform for $q\geq2/3.$  If $b_n>L^{1/2},$ then
$n\geq\exp(cL^{1/4})$ for a positive constant $c.$  The elementary estimate
\eqref{eq:prime-second-uniform}, followed by Jensen's inequality, bounds the unweighted
moment on every unit interval by $O_d(L^q).$  Hence the remaining tail is
\begin{align*}
 \ll_d L^q\sum_{n\geq\exp(cL^{1/4})}n^{-2q}
 \ll_d L^qe^{-c'L^{1/4}}
 \ll_d\left(\frac{L}{D_q(y)}\right)^q.
\end{align*}

If $0<q<2/3,$ applying Jensen's inequality to the estimate just proved at
$q=2/3$ and using $D_q(y)\asymp\sqrt{\log\log y}$ throughout this range gives
the required estimate with a constant uniform in $q,$ while the case $q=0$ is trivial.
\end{proof}

\section{Proof of Theorem~\ref{thm:main}: Upper bound}

We now pass from the Euler-product estimate to the original partial sum, following the smooth--rough decomposition in
\cite[Lemma~1.2 and Section~2]{GW2025}, with
Lemma~\ref{lem:covariance} replacing the exact orthogonality used there.

\begin{proposition}\label{prop:sum-upper}
Uniformly for sufficiently large $x$ and $0\leq q\leq1,$ we have
\begin{align*}
 \E[|S(x)|^{2q}]\ll_d\left(\frac{x}{D_q(x)}\right)^q.
\end{align*}
\end{proposition}

\begin{proof}
We first suppose that $2/3\leq q\leq1,$ and put
\begin{align*}
 \tau:=\log\log x,\qquad
 y:=\exp\left(\frac{\log x}{\tau^2}\right),\qquad
 H:=(\log x)^2.
\end{align*}
Let $\mathcal R_y$ be the set of positive integers having no prime divisor at
most $y,$ and abbreviate $T_y(u)$ to $T(u).$
Unique factorization into a $y$-smooth part and a $y$-rough part gives
\begin{align*}
 S(x)=\sum_{\substack{r\leq x\\r\in\mathcal R_y}}\X(r)T(x/r).
\end{align*}

Condition on the primes $p\leq y,$ and write $\E_{>y}$ for the
remaining expectation.  Define
\begin{align*}
 \mathcal V_x:=\E_{>y}[|S(x)|^2],
 \qquad
 \mathcal V_0:=\sum_{\substack{r\leq x\\r\in\mathcal R_y}}|T(x/r)|^2.
\end{align*}
The conditional version of \eqref{eq:covariance} gives
\begin{align*}
 \mathcal V_x-\mathcal V_0
 =\sum_{\substack{u,v\in\mathcal R_y\\u+v>2}}g(u,v)
   \sum_{\substack{r\in\mathcal R_y\\r\leq x/\max\{u,v\}}}
   T(x/(ur))\overline{T(x/(vr))}.
\end{align*}
By the Cauchy--Schwarz inequality and \eqref{eq:short-second}, we have
\begin{align*}
 \E[|\mathcal V_x-\mathcal V_0|]
 &\ll_d x(\log x)
 \sum_{\substack{u,v\in\mathcal R_y\\u+v>2}}
                 \frac{|g(u,v)|}{\sqrt{uv}}\\
 &\ll_d x(\log x)y^{-1/2}.
\end{align*}
For the last estimate, the Euler product in \eqref{eq:kernel-norm} restricted
to primes exceeding $y$ is $1+O_d(\sum_{p>y}p^{-3/2}).$
Applying conditional Jensen's inequality and then
$(u+v)^q\leq u^q+v^q$ for $u,v\geq0,$ we obtain
\begin{align*}
 \E[|S(x)|^{2q}]
 &\leq\E[\mathcal V_x^q]\\
 &\leq\E[\mathcal V_0^q]
       +\E[|\mathcal V_x-\mathcal V_0|^q]\\
 &\leq\E[\mathcal V_0^q]
       +\bigl(\E[|\mathcal V_x-\mathcal V_0|]\bigr)^q.
\end{align*}
The preceding estimate therefore gives
\begin{equation}\label{eq:conditional-upper}
 \E[|S(x)|^{2q}]
 \leq\E[\mathcal V_0^q]
 +O_d((x(\log x)y^{-1/2})^q).
\end{equation}

The part of $\mathcal V_0$ with $r<\sqrt x$ is negligible.  Indeed,
Lemma~\ref{lem:smooth-second} gives
\begin{align*}
 \E\left[\sum_{\substack{r<\sqrt x\\r\in\mathcal R_y}}
 |T(x/r)|^2\right]
 \ll_d x(\log x)(\log y)
       \exp\left(-\frac{\log x}{2\log y}\right).
\end{align*}
For $r\geq\sqrt x,$ averaging over a short interval gives
\begin{align*}
 |T(x/r)|^2
 \leq\frac{2H}{r}\int_r^{r(1+H^{-1})}|T(x/t)|^2dt
+\frac{2H}{r}\int_r^{r(1+H^{-1})}|T(x/r)-T(x/t)|^2dt.
\end{align*}
By the Selberg upper-bound sieve, uniformly for $M\geq1,$ we have
\begin{align*}
 \#\{a<n\leq a+M:n\in\mathcal R_y\}
 \ll M/\log y+y^2.
\end{align*}
Since
$\sqrt x/H\gg y^4,$ for $\sqrt x\leq t\leq2x,$ we obtain
\begin{align*}
 \sum_{\substack{t/(1+H^{-1})<r\leq t\\r\in\mathcal R_y}}
       \frac Hr
 \ll\frac1{\log y}.
\end{align*}
Changing the order of integration bounds the sum of the first integrals by
\begin{align*}
 \frac{c x}{\log y}\int_1^\infty|T(u)|^2\frac{du}{u^2}.
\end{align*}

Applying \eqref{eq:short-second} and the same sieve estimate to the error
integrals, we obtain
\begin{align*}
 \E\left[\sum_{\substack{\sqrt x\leq r\leq x\\r\in\mathcal R_y}}
 \frac Hr\int_r^{r(1+H^{-1})}|T(x/r)-T(x/t)|^2dt\right]
&\ll_d
 \sum_{\substack{\sqrt x\leq r\leq x\\r\in\mathcal R_y}}
       \left(\frac{x}{rH}+\left( \frac{x}{r} \right)^{5/6}\right)\\
 &\ll_d
 \frac{x\log x}{H\log y}+\frac{x}{\log y}
 \ll_d\frac{x}{\log y}.
\end{align*}
The penultimate estimate follows by partial summation from
$\#(\mathcal R_y\cap[1,u])\ll u/\log y$ for $u\geq\sqrt x.$
The remaining errors in \eqref{eq:conditional-upper} and in the range
$r<\sqrt x$ are $o_d(x/\log y).$  Jensen's inequality and $(u+v)^q\leq u^q+v^q$ therefore give
\begin{equation}\label{eq:upper-reduction}
 \E[|S(x)|^{2q}]
 \ll_d\left(\frac{x}{\log y}\right)^q
 \left\{\E\left[\left(\int_1^\infty|T(u)|^2\frac{du}{u^2}\right)^q\right]+1\right\}.
\end{equation}

Since $\log\log y=\tau-2\log\tau\asymp\tau,$
Lemma~\ref{lem:mellin}, Proposition~\ref{prop:general-mass-upper}, and
\eqref{eq:upper-reduction} prove the required estimate in the present range.
For $0<q<2/3,$ applying Jensen's inequality to the estimate at $q=2/3$ and
using $D_q(x)\asymp\sqrt{\log\log x}$ in this range gives the stated uniform
bound, while the endpoint $q=0$ is trivial.
\end{proof}

\section{The lower Euler-product estimate}

To prove the lower bound, we compare with Harper's Rademacher Euler product
on an interval bounded away from the origin, where the contribution of
$\rho$ to the covariance is bounded.

\begin{lemma}\label{lem:general-lower-mass}
Let $J=[1/3,1/2].$  Uniformly for all sufficiently large $y$ and all
parameters satisfying
\begin{align*}
 2/3\leq q\leq1,
 \qquad
 1\leq V\leq(\log y)^{1/100},
\end{align*}
we have
\begin{align*}
 \E\left[\left(
 \int_J|F_y(1/2+4V/\log y+it)|^2dt
 \right)^q\right]
 \gg_d\left(\frac{\log y}{V D_q(y)}\right)^q.
\end{align*}
\end{lemma}

\begin{proof}
Let
\begin{align*}
 R_y(s):=\prod_{p\leq y}(1+\varepsilon_pp^{-s}),
\end{align*}
where the $\varepsilon_p\in\{\pm1\}$ are independent Rademacher random
variables, and put $\delta=4V/\log y.$  For $s=\sigma+it,$ the exact identity
\begin{align*}
 |1+\varepsilon_pp^{-s}|^2
 =1+2p^{-\sigma}\Re(\varepsilon_pp^{-it})+p^{-2\sigma}
\end{align*}
and the moment identities
\begin{align*}
 \E[\varepsilon_p]=0,
 \qquad
 \E[|\varepsilon_p|^2]
 =\E[\varepsilon_p^2]=1
\end{align*}
give the same second-order calculation as in the proof of
Lemma~\ref{lem:gaussian-replacement}, now with $\rho=1.$  Since
$(1-p^{-\sigma})^2\leq|1+\varepsilon_pp^{-s}|^2\leq(1+p^{-\sigma})^2,$
the small primes also satisfy the required bounds, and the same proof
compares $R_y$ with the real Gaussian field
\begin{align*}
 H_{\{p:p\leq y\},1,\sigma}(t)
 =2\sum_{p\leq y}g_pp^{-\sigma}\cos(t\log p),
\end{align*}
where the $g_p$ are independent real standard Gaussian random variables.

Since $ t+t'\in [2/3,1]$ and $\delta\geq0,$ taking $\sigma=1/2+\delta,$ Lemma~\ref{lem:prime-sum} and partial summation give, uniformly for $t,t'\in J$ that
\begin{align*}
 \left|2\Re\left((\rho-1)\sum_{p\leq y}
 p^{-1-2\delta-i(t+t')}\right)\right|\ll1.
\end{align*}
This is the difference
between the covariance kernels of the normalized Gaussian fields
corresponding to $F_y(1/2+\delta+it)$ and
$R_y(1/2+\delta+it),$ while their deterministic normalizations are
identical.  Applying Lemma~\ref{lem:gaussian-replacement} to $F_y,$ the
preceding comparison to $R_y,$ and Lemma~\ref{lem:kahane} in both
directions, we obtain
\begin{equation}\label{eq:general-rad-comparison}
 \E\left[\left(
 \int_J|F_y(1/2+\delta+it)|^2dt
 \right)^q\right]
 \asymp_d
 \E\left[\left(
 \int_J|R_y(1/2+\delta+it)|^2dt
 \right)^q\right].
\end{equation}

To apply Harper's restricted Rademacher estimates, put
\begin{align*}
 \tau:=\log\log y,
 \qquad
 a:=\min\{\sqrt{\tau},(1-q)^{-1}\}.
\end{align*}
Let
$\mathcal L_a\subseteq J$ be the random set defined in
\cite[Section~5.3]{HarperLow}, with the parameter $x$ there replaced by
$y,$ and write
\begin{align*}
 M_a:=\frac{V}{\log y}\int_{\mathcal L_a}
 |R_y(1/2+\delta+it)|^2dt.
\end{align*}
The first-moment calculation in \cite[Section~5.2]{HarperLow}, using
Proposition~6 in place of Proposition~5, and the Rademacher version of
Key Proposition~5 explained in \cite[Section~5.3]{HarperLow} give
\begin{equation}\label{eq:restricted-moments}
 \E[M_a]\gg D_q(y)^{-1},
 \qquad
 \E[M_a^2]\ll e^{2a}D_q(y)^{-2}.
\end{equation}
These are precisely Harper's interval and shift, and the number of
barrier increments is $\log\log y-\log V+O(1)\asymp\tau$ throughout
$1\leq V\leq(\log y)^{1/100}.$  In particular, the first-moment factor
is $\min\{1,a/\sqrt\tau\}\asymp D_q(y)^{-1}.$

Applying H\"older's inequality and \eqref{eq:restricted-moments}, we obtain
\begin{align*}
 \E[M_a^q]
 &\geq
 \frac{\E[M_a]^{2-q}}
      {\E[M_a^2]^{1-q}}  \\
 &\gg e^{-2a(1-q)}D_q(y)^{-q}
 \gg D_q(y)^{-q},
\end{align*}
since $a(1-q)\leq1.$  As $\mathcal L_a\subseteq J,$ it follows that
\begin{align*}
 \E\left[\left(
 \int_J|R_y(1/2+\delta+it)|^2dt
 \right)^q\right]
 &\geq
 \left(\frac{\log y}{V}\right)^q
 \E[M_a^q] \\
 &\gg
 \left(\frac{\log y}{V D_q(y)}\right)^q.
\end{align*}
Combining this estimate with
\eqref{eq:general-rad-comparison} proves the lemma.
\end{proof}

\section{Proof of Theorem~\ref{thm:main}: Lower bound}

Following Harper \cite[Propositions~3--4]{HarperLow}, we condition on the primes $p\leq\sqrt{x}$ and
isolate the contribution from primes $p>\sqrt{x}.$  Since the law of
$h_1(\boldsymbol\alpha(p))$ is not necessarily invariant under multiplication
by $-1,$ we use an independent copy to remove the remaining smooth sum.

\begin{lemma}\label{lem:conditional-lower}
Let
\begin{align*}
 W_x:=\sum_{\sqrt x<p\leq x}|S(x/p)|^2.
\end{align*}
Uniformly for $0<q\leq1,$ we have
\begin{align*}
 \E[|S(x)|^{2q}]\gg_d\E[W_x^q].
\end{align*}
\end{lemma}

\begin{proof}
Condition on the random vectors $\boldsymbol\alpha(p)$ with $p\leq\sqrt x,$ and
write $\E_{>\sqrt x}$ and $\mathbb P_{>\sqrt x}$ for the remaining
expectation and probability.  Every integer at most $x$ has at most one
prime divisor exceeding $\sqrt x,$ which occurs to the first power, and hence
\begin{align*}
 S(x)=U+Z,\qquad
 U:=T_{\sqrt x}(x),\qquad
 Z:=\sum_{\sqrt x<p\leq x}h_1(\boldsymbol\alpha(p))S(x/p).
\end{align*}
Here $U$ and all the sums $S(x/p)$ are fixed under the conditioning.  Let
$\boldsymbol\alpha'(p)$ be independent copies of the random vectors with
$p>\sqrt x,$ and define
\begin{align*}
 \xi_p&:=h_1(\boldsymbol\alpha(p))-h_1(\boldsymbol\alpha'(p)),\\
 Z'&:=\sum_{\sqrt x<p\leq x}h_1(\boldsymbol\alpha'(p))S(x/p).
\end{align*}
Since $Z$ and $Z'$ have the same conditional law, the inequality
\begin{align*}
 |Z-Z'|^{2q}\leq2\bigl(|U+Z|^{2q}+|U+Z'|^{2q}\bigr)
\end{align*}
gives, uniformly for $0<q\leq1$ that
\begin{equation}\label{eq:independent-copy}
 \E_{>\sqrt x}[|S(x)|^{2q}]
 \geq\frac14\E_{>\sqrt x}[|Z-Z'|^{2q}].
\end{equation}

The random variables $\xi_p$ are independent and centered, and
\eqref{eq:conditions} gives
\begin{align*}
 \E[|\xi_p|^2]=2,\qquad \E[\xi_p^2]=2\rho,
 \qquad \E[|\xi_p|^4]\leq8d^2,
\end{align*}
where the last inequality follows from
$\E[|\xi_p|^4]\leq4d^2\E[|\xi_p|^2].$  Consequently, we have
\begin{align*}
 \E_{>\sqrt x}[|Z-Z'|^2]=2W_x.
\end{align*}
On expanding the fourth power, independence and centering eliminate all
terms in which a prime occurs only once, while the remaining terms give
\begin{gather*}
 \E_{>\sqrt x}[|Z-Z'|^4]
 =8W_x^2+4|\rho|^2
       \left|\sum_{\sqrt x<p\leq x}S(x/p)^2\right|^2\\
 +\sum_{\sqrt x<p\leq x}
       \bigl(\E[|\xi_p|^4]-8-4|\rho|^2\bigr)|S(x/p)|^4
 \leq(12+8d^2)W_x^2.
\end{gather*}
Here we used $|\rho|\leq1,$ together with
\begin{gather*}
 \left|\sum_{\sqrt x<p\leq x}S(x/p)^2\right|\leq W_x,
 \qquad
 \sum_{\sqrt x<p\leq x}|S(x/p)|^4\leq W_x^2.
\end{gather*}
For $W_x>0,$ the Paley--Zygmund inequality now gives
\begin{gather*}
 \E_{>\sqrt x}[|Z-Z'|^{2q}]
 \geq W_x^q\,\mathbb P_{>\sqrt x}(|Z-Z'|^2\geq W_x)\\
 \geq\frac{W_x^{q+2}}{\E_{>\sqrt x}[|Z-Z'|^4]}
 \geq\frac{W_x^q}{12+8d^2},
\end{gather*}
with the same conclusion immediate when $W_x=0.$  Taking expectations in
\eqref{eq:independent-copy} proves the lemma.
\end{proof}

Next, we adapt the logarithmic weighting and short-interval smoothing in
\cite[Section~2.5]{HarperLow}.

\begin{lemma}\label{lem:variance-lower}
Uniformly for sufficiently large $x$ and $2/3\leq q\leq1,$ we have
\begin{align*}
 \E[W_x^q]\gg_d\left(\frac{x}{D_q(x)}\right)^q.
\end{align*}
\end{lemma}

\begin{proof}
Put $H=(\log x)^2$ and define
\begin{align*}
 \mathcal Q_x&:=\sum_{\sqrt x<p\leq x}(\log p)|S(x/p)|^2,\\
 \mathcal M_x&:=\sum_{\sqrt x<p\leq x}\frac{H\log p}{p}
             \int_p^{p(1+H^{-1})}|S(x/t)|^2dt,\\
 \mathcal R_x&:=\sum_{\sqrt x<p\leq x}\frac{H\log p}{p}
             \int_p^{p(1+H^{-1})}|S(x/p)-S(x/t)|^2dt.
\end{align*}
The elementary inequality $|z_1|^2\leq2|z_2|^2+2|z_1-z_2|^2$ gives
\begin{align*}
 \mathcal M_x\leq2\mathcal Q_x+2\mathcal R_x.
\end{align*}
By partial summation and \eqref{eq:short-second}, we have
\begin{align*}
 \E[\mathcal R_x]
 &\ll_d\sum_{\sqrt x<p\leq x}(\log p)
       \left(\frac{x}{pH}+(x/p)^{5/6}\right)
 \\
 &\ll_d\frac{x\log x}{H}
       +x^{5/6}x^{1/6}\ll_d x.
\end{align*}

The prime number theorem with the de la Vall\'ee Poussin error term gives,
uniformly for
$2\sqrt x\leq t\leq x$ that
\begin{align*}
 \sum_{t/(1+H^{-1})<p\leq t}\frac{\log p}{p}\asymp H^{-1}.
\end{align*}
Changing the order of integration, we therefore obtain
\begin{align*}
 \mathcal M_x
 \gg\int_{2\sqrt x}^x|S(x/t)|^2dt
 =x\int_1^{\sqrt x/2}|S(u)|^2\frac{du}{u^2}.
\end{align*}
Writing the last integral as $I_x,$ we use
$\mathcal Q_x\leq(\log x)W_x,$ the inequality $(u+v)^q\leq u^q+v^q,$
and Jensen's inequality to obtain
\begin{equation}\label{eq:variance-smoothing}
 \E[W_x^q]
 \geq c\left(\frac{x}{\log x}\right)^q\E[I_x^q]
 -c_d\left(\frac{x}{\log x}\right)^q.
\end{equation}

For $\eta\geq0,$ put
\begin{align*}
 J_x(\eta):=\int_1^\infty|T_x(u)|^2\frac{du}{u^{2+2\eta}}.
\end{align*}
Given $1\leq V\leq(\log x)^{1/100},$ take
\begin{align*}
 \alpha=4V/\log x,
 \qquad
 \beta=2V/\log x,
 \qquad
 T=\sqrt x/2.
\end{align*}
Since $T_x(u)=S(u)$ for $u\leq T,$ splitting the integral at $T$ gives
\begin{align*}
 J_x(\alpha)
 &\leq I_x+T^{-2(\alpha-\beta)}J_x(\beta)\\
 &\leq I_x+e^{-V}J_x(0)
\end{align*}
for $x\geq16,$ where we used $J_x(\beta)\leq J_x(0).$  Using $(u+v)^q\leq u^q+v^q,$ we obtain
\begin{equation}\label{eq:ix-from-jx}
 \E[I_x^q]\geq\E[J_x(\alpha)^q]-e^{-Vq}\E[J_x(0)^q].
\end{equation}

By Lemma~\ref{lem:mellin} and Proposition~\ref{prop:general-mass-upper}, we
have
\begin{align*}
 \E[J_x(0)^q]\ll_d\left(\frac{\log x}{D_q(x)}\right)^q.
\end{align*}
Since $|1/2+\alpha+it|\asymp1$ on $J=[1/3,1/2],$ restricting the
Mellin integral to $J$ and applying Lemma~\ref{lem:general-lower-mass} gives
\begin{align*}
 \E[J_x(4V/\log x)^q]
 \gg_d\left(\frac{\log x}{V D_q(x)}\right)^q.
\end{align*}
Inserting these estimates into \eqref{eq:ix-from-jx}, we obtain
\begin{align*}
 \E[I_x^q]
 \geq(c_dV^{-q}-c'_de^{-Vq})
       \left(\frac{\log x}{D_q(x)}\right)^q.
\end{align*}
Choosing a sufficiently large fixed $V$ in terms of $d$ makes the
coefficient positive uniformly for $q\geq2/3,$ and substitution in
\eqref{eq:variance-smoothing} proves the lemma after absorbing the remaining
error for sufficiently large $x.$
\end{proof}

We may now complete the proof in the remaining range of moments by the same
interpolation used in \cite[Section~1.3]{HarperLow}.

\begin{proof}[Completion of the proof of Theorem~\ref{thm:main}]
The upper bound follows from Proposition~\ref{prop:sum-upper}, while
Lemma~\ref{lem:conditional-lower} and Lemma~\ref{lem:variance-lower} give
the lower bound for $2/3\leq q\leq1.$  For $0<q<2/3,$ put
\begin{align*}
 M:=\frac{x}{\sqrt{\log\log x}}.
\end{align*}
The lower bound at $q=2/3$ and the upper bound at $q=3/4$ give
\begin{align*}
 \E[|S(x)|^{4/3}]\gg_d M^{2/3},
 \qquad
 \E[|S(x)|^{3/2}]\ll_d M^{3/4}.
\end{align*}
Applying H\"older's inequality with $\theta=(9-12q)^{-1}$ and
$4/3=2q\theta+(3/2)(1-\theta),$ we obtain
\begin{align*}
 \E[|S(x)|^{2q}]
 \geq
 \frac{(\E[|S(x)|^{4/3}])^{9-12q}}
      {(\E[|S(x)|^{3/2}])^{8-12q}}
 \gg_d M^q.
\end{align*}
The powers of the implied constants are bounded uniformly in this range,
where $D_q(x)\asymp\sqrt{\log\log x},$ and the theorem follows, with
the endpoint $q=0$ immediate.
\end{proof}

\section{Further examples and the prime-square condition}
\label{sec:examples}

We first verify the assertions about compact groups by the \textit{Frobenius--Schur
trichotomy} (see, for instance, \cite[Section~2.4]{MatzTemplier}).
For a nontrivial irreducible unitary representation $V$ of a compact group
$G$ with character $\chi_V,$ the identities
\begin{align*}
 h_1(\boldsymbol\alpha)=\chi_V,
 \qquad
 h_2(\boldsymbol\alpha)=\chi_{\operatorname{Sym}^2V},
\end{align*}
and orthogonality relations give the three expectations stated there.
If $V$ is non-self-dual, which is called \textit{unitary type}, then
$(V\otimes V)^G=0.$  Otherwise, the invariant bilinear form is unique up to
scaling and is either alternating or symmetric, corresponding to \textit{symplectic}
or \textit{orthogonal type}, respectively.  Therefore, we have
\begin{align*}
 \E[h_2(\boldsymbol\alpha)]=\dim(\operatorname{Sym}^2V)^G
 =\begin{cases}
 0&\text{if $V$ is of unitary or symplectic type},\\
 1&\text{if $V$ is of orthogonal type},
 \end{cases}
\end{align*}
which verifies the claimed scope of Theorem~\ref{thm:main}.

For $V=\operatorname{Sym}^m(\C^2)$ as in Example~\ref{ex:odd-sym}, the invariant form induced by the
$m$-fold tensor power of the determinant pairing has symmetry $(-1)^m.$
Its restriction to $V$ pairs $e_1^m$ and $e_2^m$ nontrivially and is
therefore nondegenerate by irreducibility, which gives symplectic type for
odd $m$ and orthogonal type for even $m\geq2.$

The symmetric-square case shows that the final condition in
\eqref{eq:conditions} cannot simply be discarded while retaining the same
conclusion.

\begin{proposition}\label{prop:sym-square}
Let $\X_2$ be the random multiplicative function with local parameters $ e^{2i\theta_p}, 1, e^{-2i\theta_p},$ where the $\theta_p \in [0,\pi]$ are independent Sato--Tate angles. Then
\begin{align*}
 \E\left[\left|\sum_{n\leq x}\X_2(n)\right|^2 \right]=x\log x+O(x),
 \end{align*}
 and
 \begin{align*}
 \E\left[\sum_{n\leq x}\X_2(n) \right]=\lfloor\sqrt x\rfloor.
\end{align*}
\end{proposition}

\begin{proof}
Write $h_k=h_k(e^{2i\theta},1,e^{-2i\theta})$ and
$\chi_m(\theta)=\sin((m+1)\theta)/\sin\theta.$ For $|z|<1,$ we have
\begin{gather*}
 \sum_{r\geq0}\chi_{2r}(\theta)z^r
 =\frac{1+z}{(1-e^{2i\theta}z)(1-e^{-2i\theta}z)},\\
 \sum_{k\geq0}h_kz^k
 =\frac1{1-z^2}\sum_{r\geq0}\chi_{2r}(\theta)z^r,
 \qquad \text{where }
 h_k=\sum_{0\leq j\leq k/2}\chi_{2k-4j}(\theta).
\end{gather*}
Character orthogonality now gives
\begin{align*}
 \sum_{k,l\geq0}\E[h_kh_l]z^kw^l
 &=\sum_{r,j_1,j_2\geq0}z^{r+2j_1}w^{r+2j_2}\\
 &=\frac1{(1-zw)(1-z^2)(1-w^2)}.
\end{align*}
Independence over the primes therefore gives
\begin{align*}
 \E[\X_2(m)\X_2(n)]
 =\#\{(r,u,v)\in\N^3:m=ru^2,\ n=rv^2\}.
\end{align*}
Summing over $m,n\leq x,$ we obtain
\begin{align*}
 \E\left[\left|\sum_{n\leq x}\X_2(n)\right|^2\right]
 &=\sum_{r\leq x}\left\lfloor\sqrt{\frac{x}{r}}\right\rfloor^2\\
 &=x\sum_{r\leq x}\frac1r
   +O\left(\sqrt x\sum_{r\leq x}r^{-1/2}\right)\\
 &=x\log x+O(x).
\end{align*}
The character identity gives $\E[h_k]=1$ for even $k$ and $\E[h_k]=0$ for odd
$k.$  Hence $\E[\X_2(n)]$ is the indicator function of the squares, which
proves the first-moment formula.
\end{proof}

In particular, it rules out
the order of magnitude in Theorem~\ref{thm:main} at $q=1/2$
and $q=1$, respectively.
More generally, suppose only that $\E[h_1(\boldsymbol\alpha)]=0,$ and put
\begin{align*}
 v:=\E[|h_1(\boldsymbol\alpha)|^2],
 \qquad
 \rho:=\E[h_1(\boldsymbol\alpha)^2],
 \qquad
 \kappa:=\E[h_2(\boldsymbol\alpha)].
\end{align*}
Uniformly for $0\leq q\leq1,\sigma\geq1/2,t\in\R,$ a local
Taylor expansion gives
\begin{equation*}
 \log\E[|F_y(\sigma+it)|^{2q}]
 =q^2v\sum_{p\leq y}p^{-2\sigma}
+\Re\left\{(2q\kappa+q(q-1)\rho)
       \sum_{p\leq y}p^{-2\sigma-2it}\right\}+O_d(1).
\end{equation*}
Indeed, the analytic branch at the origin satisfies
\begin{align*}
 \left(\sum_{k\geq0}h_kz^k\right)^q
 =1+qh_1z+\left(qh_2+\frac{q(q-1)}2h_1^2\right)z^2
 +O_d(|z|^3).
\end{align*}
After multiplying by the conjugate series, taking expectations and then
logarithms, the errors are absolutely summable over the primes.

The Steinhaus RMF is the case $d=1$ of
Example~\ref{ex:steinhaus}, whereas a Rademacher RMF has local factor
$1+\varepsilon_pz,$ which is not an inverse Euler factor with unitary
parameters and is therefore not automorphic in the sense of
Definition~\ref{def:armf}.  Its low moments nevertheless have the same order
by \cite[Theorem~2]{HarperLow}.  The extended Rademacher RMF is the
orthogonal model attached to $O(1)$,\footnote{Here orthogonal refers to the representation of $O(1)$,
not to the random-matrix symmetry of the quadratic Dirichlet
$L$-function family, which is symplectic (see, for instance, \cite[Section~4.4]{MR2149530}).}
for which $h_1=\varepsilon_p$ and $h_2=1,$ and hence it satisfies the first
two conditions in \eqref{eq:conditions}, but not the third.

\section*{Acknowledgements}

The author would like to thank Jad Hamdan and Mo Dick Wong for helpful discussions, and acknowledges the use of ChatGPT for performing sanity checks and for rephrasing. This work is supported by the Croucher Fellowship for Postdoctoral Research.

\printbibliography

\end{document}